\documentclass[11pt, a4paper, oneside]{amsart}

\usepackage[english]{babel}
\usepackage{amsmath, amsthm, amsfonts, mathrsfs, amssymb}
\usepackage{mathtools}
\mathtoolsset{centercolon}
\usepackage{mathrsfs}
\usepackage{amscd}
\usepackage{bm}

\usepackage{tikz-cd}

\usepackage[shortlabels]{enumitem}
\usepackage[colorlinks, citecolor = blue, urlcolor={red}]{hyperref}
\usepackage[colorinlistoftodos,prependcaption,textsize=small]{todonotes}

\usepackage{faktor}

\allowdisplaybreaks

\setlist[itemize]{leftmargin=25pt}
\setlist[enumerate]{leftmargin=25pt}

\usepackage{color}

\newtheorem{theorem}{Theorem}
\newtheorem{corollary}[theorem]{Corollary}
\newtheorem{lemma}[theorem]{Lemma}
\newtheorem{proposition}[theorem]{Proposition}

\theoremstyle{remark}
\newtheorem{remark}[theorem]{Remark}

\theoremstyle{definition}

\numberwithin{theorem}{section}
\numberwithin{equation}{section}

\def\N{{\mathbb N}}
\def\Z{{\mathbb Z}}

\def\R{{\mathbb R}}
\def\C{{\mathbb C}}

\renewcommand{\P}{{\mathbb P}}

\DeclareMathAlphabet{\mathpzc}{OT1}{pzc}{m}{it}

\newcommand{\supp}{\text{\rm supp\,}}

\DeclareFontFamily{U}{mathx}{\hyphenchar\font45}
\DeclareFontShape{U}{mathx}{m}{n}{<5> <6> <7> <8> <9> <10> <10.95> <12> <14.4> <17.28> <20.74> <24.88> mathx10}{}
\DeclareSymbolFont{mathx}{U}{mathx}{m}{n}
\DeclareFontSubstitution{U}{mathx}{m}{n}
\DeclareMathAccent{\widecheck}{0}{mathx}{"71}

\newcommand{\mc}{\mathcal}
\newcommand{\ms}{\mathscr}

\DeclarePairedDelimiter\abs{\lvert}{\rvert}

\DeclarePairedDelimiter\cbrace\{\}
\DeclarePairedDelimiter\ha()

\DeclarePairedDelimiter{\nrm}\lVert\rVert

\newcommand{\nrmb}[1]{\bigl\|#1\bigr\|}

\newcommand{\hab}[1]{\bigl(#1\bigr)}
\newcommand{\cbraceb}[1]{\bigl\{#1\bigr\}}

\newcommand{\nrms}[1]{\Bigl\|#1\Bigr\|}

\newcommand{\has}[1]{\Bigl(#1\Bigr)}

\newcommand{\dd}{\hspace{2pt}\mathrm{d}}
\newcommand{\ddn}{\mathrm{d}}
\newcommand{\ee}{\mathrm{e}}

\DeclareMathOperator{\ind}{\mathbf{1}}

\DeclareMathOperator{\re}{Re}

\begin{document}

\title{Stein interpolation for the real interpolation method}

\author{Nick Lindemulder}
\address[N. Lindemulder]{Institute of Analysis \\
Karlsruhe Institute of Technology \\
Englerstra{\ss}e 2 \\
76131 Karlsruhe\\
Germany}
\email{nick.lindemulder@kit.edu}

\author{Emiel Lorist}
\address[E. Lorist]{Delft Institute of Applied Mathematics\\
Delft University of Technology \\ P.O. Box 5031\\ 2600 GA Delft\\The
Netherlands} \email{emiellorist@gmail.com}


%46B70 (1991-now) Interpolation between normed linear spaces
%46M35 (1973-now) Abstract interpolation of topological vector spaces


\thanks{The second author was supported by the Vidi subsidy 639.032.427 of the Netherlands Organisation for Scientific Research (NWO)}

\begin{abstract}
We prove a complex formulation of the real interpolation method, showing that the real and complex interpolation methods are not inherently real or complex. Using this complex formulation, we prove Stein interpolation for the real interpolation method. We apply this theorem to interpolate weighted $L^p$-spaces and the sectoriality of closed operators with the real interpolation method.
\end{abstract}

\maketitle

\section{Introduction}
In \cite{St56} Stein proved a convexity principle for the interpolation of analytic operator families on $L^p$-spaces. An important special case of \cite[Theorem 1]{St56} states that an analytic family of linear operators $\cbrace{T(z)}_{z \in \overline{\mathbb{S}}}$ which satisfies
\begin{align*}
  \nrm{T\ha{j+it}f}_{L^{p_j}(S)} &\leq M_j \,\nrm{f}_{L^{q_j}(S)}, & &t \in \R, \, j=0,1
\intertext{for any simple function $f$ and $p_0,p_1,q_0,q_1 \in [1,\infty]$ also satisfies}
  \nrm{T\ha{\theta}f}_{L^p{(S)}} &\leq M_0^{1-\theta}M_1^\theta\,\nrm{f}_{L^{q}(S)},
\end{align*}
where $\theta \in (0,1)$, $\frac{1}{p} = \frac{1-\theta}{p_0}+\frac{\theta}{p_1}$ and $\frac{1}{q} = \frac{1-\theta}{q_0}+\frac{\theta}{q_1}$.
After the development of the complex interpolation method by Calder\'on \cite{Ca64}, this theorem was generalized to general interpolation couples of (quasi)-Banach spaces, see e.g. \cite{CJ84,CS88,Vo92}.
 In \cite{SW06} Stein interpolation was proved for the so-called $\gamma$-interpolation method, using the complex formulation of the $\gamma$-interpolation method in \cite{KLW19}.
% These abstract Stein interpolation theorems are very useful when one wants to interpolate e.g. weighted $L^p$-spaces or sectorial operators.

The main goal of this paper is to develop Stein interpolation for the real interpolation method.
Since Stein interpolation is inherently tied to complex function theory, this requires a complex formulation of the real interpolation method. In \cite{Cw78,Pe71} Cwikel and Peetre developed a real formulation of the complex interpolation method modelled after the Lions--Peetre mean method for real interpolation \cite{LP64}. Our first main result is a complex formulation of the Lions--Peetre mean method. Combined with the result of Cwikel and Peetre this shows that the real and complex interpolation methods are not inherently real or complex. These interpolation methods are rather living on opposite sides of the Fourier transform. In the forthcoming papers \cite{LL21b, LL21a} we will push this viewpoint further by introducing an abstract framework, containing the real and complex interpolation methods and which has both a real and a complex formulation.

In order to state the complex formulation of the Lions--Peetre mean method, define the strip
\begin{equation*}
  \mathbb{S}:= \cbrace{z \in \C: 0<\re(z)<1}.
\end{equation*}
For a Banach space  $X$ let
$\ms{H}(\overline{\mathbb{S}};X)$
be the space of all continuous functions $f \colon \overline{\mathbb{S}} \to X$ which are analytic on $\mathbb{S}$
 and let $\ms{H}^1(\overline{\mathbb{S}};X)$ be the subspace of all $f \in \ms{H}(\overline{\mathbb{S}};X)$ which satisfy
\begin{equation*}
  \sup_{s \in [0,1]} \nrm{t\mapsto f(s+it)}_{L^1(\R;X)}<\infty.
\end{equation*}
For $f \in \ms{H}^1(\overline{\mathbb{S}};X)$ we define
\begin{align*}
f_j(t)&:= f(j+it), && t \in \R, \, j=0,1.
\end{align*}For $g \in L^1(\R;X)$ we use the Fourier transform and its inverse
\begin{align*}
\widehat{g}(\xi)&:=
   \int_{\R} g(t)e^{-it\xi}\dd t, &\widecheck{g}(t)&:=
   \int_{\R} g(\xi)e^{it\xi}\dd \xi,
\end{align*}
which yields a factor $2\pi$ in the Fourier inversion formula.
For an interpolation couple of Banach spaces $(X_0,X_1)$, $\theta \in (0,1)$ and $p \in [1,\infty]$ we denote the real interpolation spaces by $(X_0,X_1)_{\theta,p}$.

\begin{theorem}\label{theorem:complexintro}
  Let $(X_0,X_1)$ be an interpolation couple of Banach spaces, let $p_0,p_1 \in [1,\infty]$ and let $\theta \in (0,1)$. Set $\frac{1}{p} = \frac{1-\theta}{p_0}+ \frac{\theta}{p_1}$ and define
\begin{align*}
  \nrm{x}_{({X}_0,{X}_1)_{\theta,p_0,p_1}}^{(\mathrm{c})} &:= \inf\,\max_{j = 0,1}\,\nrm{\widehat{f}_j}_{L^{p_j}(\R;X_j)}, && x \in X_0+X_1,
\end{align*}
where the infimum is taken over all $f\in \ms{H}^1(\overline{\mathbb{S}};X_0+X_1)$ with $f(\theta) =x$.
Then we have
\begin{align*}
  \nrm{x}_{({X}_0,{X}_1)_{\theta,p}} \lesssim_{\theta} \nrm{x}_{({X}_0,{X}_1)_{\theta,p_0,p_1}}^{(\mathrm{c})}, &&x \in X_0+X_1,\\
   \nrm{x}_{({X}_0,{X}_1)_{\theta,p}} \gtrsim_{\theta}  \nrm{x}_{({X}_0,{X}_1)_{\theta,p_0,p_1}}^{(\mathrm{c})}, &&x \in X_0\cap X_1.
\end{align*}
\end{theorem}
Theorem \ref{theorem:complexintro} is a direct consequence of Theorem \ref{theorem:complex} and the equivalence of the real and Lions--Peetre mean methods.
Note that it only yields an equivalent norm for $x \in X_0\cap X_1$. One can not directly extend this to all $x \in ({X}_0,{X}_1)_{\theta,p}$ using density, since
$$
(X_0,X_1)_{\theta,p_0,p_1}^{(\mathrm{c})}:= \cbraceb{x \in X_0+X_1: \nrm{x}_{({X}_0,{X}_1)_{\theta,p_0,p_1}}^{(\mathrm{c})}<\infty}
$$
is not a Banach space. One can circumvent this issue by considering appropriate spaces of distributions, which will be further explored in \cite{LL21b,LL21a}. In applications the norm equivalence for $x \in X_0 \cap X_1$ usually suffices, avoiding the additional technicalities that distribution theory brings.

The idea to use analytic functions for the real interpolation method is not new. Implicitly it goes back to the work of Lions and Peetre \cite[Chapitre IV]{LP64} and a rudimentary version of Theorem \ref{theorem:complexintro} was used by Zafran in \cite{Za80} to prove \u{S}ne\u{i}berg's lemma for the real interpolation method. Moreover Cwikel, Kalton, Milman and Rochberg \cite{CKMR02} build an unified theory for a class of interpolation methods , containing the real interpolation method, in which they systematically use analytic functions to prove commutator estimates. The key difference between Theorem \ref{theorem:complexintro} and these prior works is the fact that we express the norm of $({X}_0,{X}_1)_{\theta,p_0,p_1}$ in terms of the boundary functions of $f_j$ for $j=0,1$, rather than in terms of the function $f(\theta + i\cdot)$. As a result we have a genuine complex interpolation method formulation of the real interpolation method, which is key in order to prove  the announced version of Stein interpolation for the real interpolation method.

\begin{theorem}\label{theorem:Steininterpolationintro}
Let $(X_0,X_1)$ and $(Y_0,Y_1)$ be interpolation couples of Banach spaces and let $\breve{X}$ be a dense subspace of $X_0\cap X_1$. Let $p_0,p_1,q_0,q_1 \in [1,\infty]$, set $$\frac{1}{p} = \frac{1-\theta}{p_0}+\frac{\theta}{p_1}, \qquad \frac{1}{q} = \frac{1-\theta}{q_0}+\frac{\theta}{q_1},$$
and suppose $p<\infty$. Let $\cbrace{T(z)}_{z \in \overline{\mathbb{S}}}$ be a family of linear operators from $\breve{X}$ to $ Y_0+Y_1$ such that
\begin{enumerate}[(i)]
  \item\label{it:stein1intro} $ T(\cdot) x\in \ms{H}(\overline{\mathbb{S}};Y_0+Y_1)$ for all $x \in \breve{X}$.
  %\item \label{it:stein1b} $\lim_{s \to j} T_{s+it}x = T_{j+it}x$ in $Y_0+Y_1$ for all $x \in \breve{X}$ and $j=0,1$.
  \item \label{it:stein2intro} $T_j x:= T(j+i\cdot)x \in L^\infty(\R;Y_j)$  for all $x \in \breve{X}$ and
  \begin{equation*}
    \nrmb{(T_j \widehat{f}\hspace{2pt})^{\vee}}_{L^{q_j}(\R;Y_j)} \leq M_j \, \nrm{f}_{L^{p_j}(\R;X_j)}, \qquad f \in C_c^\infty(\R)\otimes \breve{X}.
  \end{equation*}
  for some $M_j>0$ and $j=0,1$.
\end{enumerate}
Then $T(\theta)$ is bounded from $(X_0,X_1)_{\theta,p}$ to  $(Y_0,Y_1)_{\theta,q}$ for any $\theta \in (0,1)$ with
\begin{equation*}
  \nrm{T(\theta)}_{(X_0,X_1)_{\theta,p} \to (Y_0,Y_1)_{\theta,q}} \lesssim_{\theta}  M_0^{1-\theta} M_1^\theta.
\end{equation*}
\end{theorem}

Theorem \ref{theorem:Steininterpolationintro} is a direct consequence of Theorem \ref{theorem:Steininterpolation1}, the equivalence of the real and Lions--Peetre mean methods and the density of $X_0\cap X_1$ in  $(X_0,X_1)_{\theta,p}$ when $p<\infty$. Assumption \ref{it:stein2intro} in Theorem \ref{theorem:Steininterpolationintro} is a Fourier multiplier condition, whereas Stein interpolation for the complex interpolation method can be seen as a pointwise multiplier condition. This difference is due to the previously observed fact that the real and complex interpolation methods live on opposite sides of the Fourier transform.

The Fourier multiplier condition \ref{it:stein2intro} in Theorem \ref{theorem:Steininterpolationintro} can in various situation be reduced to a simpler condition on $m$. Indeed, let $1 \leq p \leq q \leq \infty$. For $m \colon \R \to \mc{L}(X,Y)$ one has
  \begin{equation*}
    \nrmb{(m \widehat{f}\hspace{2pt})^{\vee}}_{L^{q}(\R;Y)} \lesssim\nrm{f}_{L^{p}(\R;X)}, \qquad f \in L^{p}(\R;X),
  \end{equation*}
under either of the following conditions:
\begin{itemize}
  \item \textit{Smoothness:} $p=q$ and $m \in C^2(\R;\mc{L}(X,Y))$ with for some $\varepsilon>0$
  \begin{equation*}
    \nrmb{\frac{\dd^n}{\dd t^n}m(\xi)}_{\mc{L}(X,Y)} \lesssim \frac{1}{1+\abs{\xi}^{n+\varepsilon}} \qquad \xi \in \R,\,n=0,1,2,
  \end{equation*}
  since in this case $\widecheck{m} \in L^1(\R;\mc{L}(X,Y))$ (see \cite[Corollary 4.4]{Am97}). Note that if $T_j$ in Theorem  \ref{theorem:Steininterpolationintro} is of polynomial growth at infinity, the assumed decay is easily obtained by multiplying $T(z)$ with $\ee^{(z-\theta)^2}$.
  \item \textit{Banach space geometry:} $X$ has Fourier type $p$ and $Y$ has Fourier type $q'$ and $m \colon \R \to \mc{L}(X,Y)$ is strongly measurable in the strong operator topology with $\xi \mapsto \nrm{m(\xi)}_{\mc{L}(X,Y)} \in L^{r}(\R)$
        for $\frac{1}{r} = \frac1p-\frac1q$, see \cite[Proposition 3.9]{RV18}. One can weaken the assumption on $m$ to the weak space $L^{r,\infty}(\R)$ under a Fourier type $p_0>p$ assumption on $X$ and a Fourier type $q_0<q'$ assumption on $Y$, see \cite[Theorem 3.12]{RV18}. Further results under type and cotype, rather than Fourier type, can also be found in \cite{RV18}.
\end{itemize}

We will give two applications of Theorem \ref{theorem:Steininterpolationintro}. Firstly we will deduce interpolation for weighted, vector-valued $L^p$-spaces from the unweighted case in Proposition \ref{proposition:weightedLp}, for which the Fourier multiplier condition can be checked using Fubini's theorem. Secondly we will interpolate the angle of ($\mc{R}$-)sectorial operators Proposition \ref{proposition:interpolationsectorial}, for which the Fourier multiplier condition can be checked using a simpler version of the smoothness condition noted above. As a consequence we will improve a result of Fackler \cite{Fa13b} on the interpolation of ($\mc{R}$-)analytic semigroups.

\bigskip

This article is organized as follows: In Section \ref{section:LP} we prove some preliminary approximation results for the Lions--Peetre mean method. In Section \ref{section:complex} we will give a complex formulation of the Lions--Peetre interpolation method and in Section \ref{section:stein} we will prove Stein interpolation for the Lions--Peetre interpolation method. Finally, in Section \ref{section:examples} we will deduce the announced applications.

\subsection*{Notation}
 By $\lesssim_{a,b,\ldots}$ we mean that there is a constant $C>0$ depending on $a,b,\ldots$ such that inequality holds and by $\eqsim_{a,b,\ldots}$ we mean that $\lesssim_{a,b,\ldots}$ and $\gtrsim_{a,b,\ldots}$ hold.

\section{The Lions--Peetre mean method}\label{section:LP}
For a general background on interpolation theory we refer to \cite{BL76, Tr78} or \cite[Appendix C]{HNVW16}. For an interpolation couple of Banach spaces $(X_0,X_1)$, $t>0$ and $x \in X_0+X_1$ the $K$-functional is given by
\begin{equation*}
  K(t,x,X_0,X_1) = \inf\cbraceb{\nrm{x_0}_{X_0}+t\nrm{x_1}_{X_1}:x = x_0+x_1}
\end{equation*}
and for $\theta \in (0,1)$ and $p \in [1,\infty]$ the real interpolation space $(X_0,X_1)_{\theta,p}$ is given as all $x \in X_0+X_1$ such that
\begin{equation*}
  \nrm{x}_{(X_0,X_1)_{\theta,p}} := \nrm{t \mapsto t^{-\theta}K(t,x,X_0,X_1)}_{L^p(\R,\frac{\ddn t}{t})} <\infty.
\end{equation*}
In \cite{LP64} Lions and Peetre introduced their mean methods. For $p_0,p_1 \in [1,\infty]$ they defined $(X_0,X_1)_{\theta,p_0,p_1}$ as the space of all $x \in X_0+X_1$ such that
\begin{equation*}
    \nrm{x}_{(X_0,X_1)_{\theta,p_0,p_1}} := \inf \max_{j=0,1} \,\nrm{t\mapsto \ee^{t(j-\theta)}f(t)}_{L^{p_j}(\R;X_j)}<\infty,
  \end{equation*}
  where the infimum is taken over all strongly measurable $f\colon \R \to X_0\cap X_1$ with $\int_{\R} f(t)\dd t = x$. Note that $\int_{\R} f(t)\dd t$  converges as a Bochner integral in $X_0+X_1$ by the $L^{p_j}$-bounds. Lions and Peetre showed   that
  $$
  (X_0,X_1)_{\theta,p} = (X_0,X_1)_{\theta,p_0,p_1}
  $$
isomorphically when $\frac{1}{p} = \frac{1-\theta}{p_0}+\frac{\theta}{p_1}$, so the Lions--Peetre mean method is a reformulation of the real interpolation method.

In the upcoming sections we will need to be able to restrict to smooth functions in the definition of the Lions--Peetre mean method, for which we will prove the following lemma. For a Banach space $X$ we denote the $X$-valued Schwartz functions by $\ms{S}(\R;X)$.

\begin{lemma}\label{lemma:schwartzenough}
  Let $(X_0,X_1)$ be an interpolation couple of Banach spaces, let $p_0,p_1 \in [1,\infty]$ and let $\theta \in (0,1)$.
  \begin{enumerate}[(i)]
    \item \label{it:schwartzLionPeetre} For $x \in X_0+X_1$ we have
  \begin{equation*}
    \nrm{x}_{(X_0,X_1)_{\theta,p_0,p_1}} \eqsim \inf \max_{j=0,1} \,\nrm{t\mapsto \ee^{t(j-\theta)}f(t)}_{L^{p_j}(\R;X_j)},
  \end{equation*}
  where the infimum is taken over $f \in \ms{S}(\R;X_0+X_1)$ with $\int_{\R} f(t)\dd t = x$.
    \item \label{it:CcLionPeetre} Let $\breve{X}$ be a dense subset of $X_0\cap X_1$. For $x \in \breve{X}$ we have
  \begin{equation*}
    \nrm{x}_{(X_0,X_1)_{\theta,p_0,p_1}} \eqsim_\theta \inf \max_{j=0,1} \,\nrm{t\mapsto \ee^{t(j-\theta)}f(t)}_{L^{p_j}(\R;X_j)},
  \end{equation*}
  where the infimum is taken over $f \in C^\infty_c(\R)\otimes \breve{X}$ with $\int_{\R} f(t)\dd t = x$.
  \end{enumerate}

\end{lemma}

\begin{proof}
  The inequality ``$\leq$'' follows in both cases directly from the definition of the Lions--Peetre mean method. Take $x \in {(X_0,X_1)_{\theta,p_0,p_1}}$ with ${\nrm{x}_{(X_0,X_1)_{\theta,p_0,p_1}}} =1$ and let $f \colon \R \to X_0\cap X_1$ be strongly measurable such that $\int_\R f(t)\dd t = x$ and
  \begin{equation*}
    \max_{j=0,1} \,\nrm{t\mapsto \ee^{t(j-\theta)}f(t)}_{L^{p_j}(\R;X_j)}\leq 2.
  \end{equation*}
    For \ref{it:schwartzLionPeetre} define $$x_k := \int_{k}^{k+1} f(t)\dd t, \qquad k \in \Z$$ and let $\varphi \in C_c^\infty(\R)$ with $\int_\R\varphi(t)\dd t=1$, $\supp \varphi\subseteq [0,1]$ and $\nrm{\varphi}_{L^\infty(\R)} \leq 2$. Setting $$\widetilde{f}(t):= \sum_{k \in \Z} x_k \otimes \varphi(t -k), \qquad t \in \R,$$ we have
  \begin{equation*}
    \max_{j=0,1} \,  \nrm{t\mapsto \ee^{t(j-\theta)}\widetilde{f}(t)}_{L^{p_j}(\R;X_j)} \lesssim \max_{j=0,1} \, \nrm{t\mapsto \ee^{t(j-\theta)}f(t)}_{L^{p_j}(\R;X_j)} \leq 2
  \end{equation*}
  and $\int_\R \widetilde{f}(t)\dd t = x$. Moreover
  \begin{equation*}
    \nrm{x_k}_{X_0+X_1} \lesssim \min\cbrace{\ee^{k\theta}, \ee^{-k(1-\theta)}}, \qquad k \in \Z,
  \end{equation*}
  so $\widetilde{f} \in \ms{S}(\R;X_0+X_1)$, which proves \ref{it:schwartzLionPeetre}.

   For \ref{it:CcLionPeetre} assume additionally that $x \in \breve{X}$. Take $n \in \N$ such that $\nrm{x}_{X_0\cap X_1} \leq n $ and
%  \begin{align*}
%   \nrm{f\ind_{(-\infty,-n)}}_{L^{p_0}(\R, \ee^{-p_0t\theta};X_0)}&\leq p_0\theta \varepsilon  \\
%    \nrm{f\ind_{(n,\infty)}}_{L^{p_1}(\R, \ee^{p_1t(1-\theta)};X_j)}&\leq p_1(1-\theta)\varepsilon.
%  \end{align*}
  define
  \begin{align*}
    y_+ &:= \int_n^\infty f(t) \dd t = x - \int_{-\infty}^n f(t) \dd t
 %   y_- &:= \int_{-\infty}^{-n} f(t) \dd t = x - \int_{-n}^\infty f(t) \dd t
  \end{align*}
  By H\"older's inequality we have
  \begin{align*}
    \nrm{y_+}_{X_0} &\leq \nrm{x}_{X_0} + \int_{-\infty}^n \nrm{f(t)}_{X_0}\dd t \leq n+ \frac{2}{(\theta p_{0}')^{1/p_0'}}\cdot \ee^{n\theta} \leq n+ \frac{2}{\theta}\cdot \ee^{n\theta}\\
     \nrm{y_+}_{X_1} &\leq \int_{n}^{\infty} \nrm{f(t)}_{X_1}\dd t \leq \frac{1}{((1-\theta)p_1')^{1/p_1'}}\cdot \ee^{-n(1-\theta)}\leq \frac{1}{1-\theta}\cdot \ee^{-n(1-\theta)}.
    %\nrm{y_-}_{X_0} &\leq \int_{-\infty}^{-n} \nrm{f(t)}_{X_0}\dd t \leq \ee^{-p_0n\theta} \nrm{x}_{(X_0,X_1)_{\theta,p_0,p_1}}
  \end{align*}
  so in particular $y_+\in X_0 \cap X_1$.
Therefore, $g_+:= \ind_{[n-1,n]} \otimes y_+$ satisfies
  \begin{equation*}
    \max_{j=0,1} \, \nrm{t\mapsto \ee^{t(j-\theta)}g_+(t)}_{L^{p_j}(\R;X_j)} \lesssim_\theta 1.
  \end{equation*}
  Analogously, $g_-:= \ind_{[-n,-n+1]} \otimes \int_{-\infty}^{-n} f(t) \dd t$  satisfies the same estimate, so defining $g:= f \ind_{[-n,n]} + g_++g_-$ we have $\int_\R g(t) \dd t = x$ and
      \begin{align*}
    \max_{j=0,1} \, \nrm{t\mapsto \ee^{t(j-\theta)}g(t)}_{L^{p_j}(\R;X_j)} \lesssim_\theta 1.
  \end{align*}
Similar to the proof of \ref{it:schwartzLionPeetre}
define $$y_k := \int_{k}^{k+1} g(t)\dd t, \qquad -n\leq k \leq n-1$$ and set $$\widetilde{g}(t):= \sum_{k = -n}^{n-1} y_k \otimes \varphi(t -k), \qquad t \in \R.$$ As before we have
  \begin{equation*}
    \nrm{t\mapsto \ee^{t(j-\theta)}\widetilde{g}(t)}_{L^{p_j}(\R;X_j)} \lesssim \max_{j=0,1} \, \nrm{t\mapsto \ee^{t(j-\theta)}g(t)}_{L^{p_j}(\R;X_j)} \lesssim_\theta 1.
  \end{equation*}
  Let $(z_k)_{k=-n}^{n-1}\subseteq \breve{X}$ be such that $\nrm{y_k-z_k}_{X_0\cap X_1} \leq {\ee^{-n}}$ and define
  $$\widetilde{h}(t) := \sum_{k = -n}^{n-1} z_k \otimes \varphi(t -k), \qquad t \in \R$$ with $\varphi$ as before. Then we have $$z:= \sum_{k=-n}^{n-1}y_k-z_k = x - \sum_{k=-n}^{n-1} z_k \in \breve{X},$$
  so $h := \widetilde{h} + \varphi \otimes z$ is a function in $C^\infty_c(\R) \otimes \breve{X}$ and $\int_\R h(t)\dd t=x$. Moreover, since $\nrm{\widetilde{g}-\widetilde{h}}_{L^\infty(\R;X_0\cap X_1)}\leq 2{\ee^{-n}}$ and $\nrm{z}_{X_0\cap X_1} \leq 1$ by construction, we have
  \begin{align*}
    \max_{j=0,1} \,\nrm{t\mapsto \ee^{t(j-\theta)}h(t)}_{L^{p_j}(\R;X_j)} &\lesssim_\theta 1
  \end{align*}
  which finishes the proof.
\end{proof}

\section{A complex formulation of the Lion--Peetre mean method}\label{section:complex}
We now turn to the complex formulation of the Lions--Peetre mean method, and thus of the real interpolation method. Let $X$ be a Banach space. In addition to the spaces $\ms{H}(\overline{\mathbb{S}};X)$ and $\ms{H}^1(\overline{\mathbb{S}};X)$ defined in the introduction, we will use spaces of holomorphic functions which are not necessarily continuous on the boundary. Let
$\ms{H}({\mathbb{S}};X)$
be the space of all analytic functions $f \colon \mathbb{S} \to X$
 and let $\ms{H}^1({\mathbb{S}};X)$ be the subspace of all $f \in \ms{H}({\mathbb{S}};X)$ which satisfy
\begin{equation*}
  \sup_{s \in (0,1)} \nrm{t\mapsto f(s+it)}_{L^1(\R;X)}<\infty.
\end{equation*}
For $f \in \ms{H}({\mathbb{S}};X)$ we define
\begin{align*}
f_s(t)&:= f(s+it), && t \in \R, \, s \in (0,1)
\end{align*}
and for $f \in \ms{H}^1({\mathbb{S}};X)$ the limits $f_j:= \lim_{s \to j} f_s$ exist in $L^1(\R;X)$ (see e.g. \cite{BK07,Pi16}). Note that these limits  coincide with the pointwise limit if $f \in \ms{H}^1(\overline{\mathbb{S}};X)$.

Our first complex formulation is a simple reformulation of the Lions-Peetre mean method using the Fourier transform of $f_\theta$.

\begin{proposition}\label{proposition:complexsimple}
  Let $(X_0,X_1)$ be an interpolation couple of Banach spaces, let $p_0,p_1 \in [1,\infty]$ and let $\theta \in (0,1)$.
For $x \in X_0+X_1$ we have
\begin{align*}
  \nrm{x}_{({X}_0,{X}_1)_{\theta,p_0,p_1}} &\eqsim \inf\,\max_{j = 0,1}\,\nrm{t \mapsto \ee^{t(j-\theta)} \widehat{f}_\theta(t)}_{L^{p_j}(\R;X_j)},
\end{align*}
where the infimum is taken over all $f \in \ms{H}(\mathbb{S};X_0+X_1)$ with $f(\theta) =x$ and  $f_\theta \in L^1(\R;X_0+X_1)$.
\end{proposition}

\begin{proof}
First let $x \in X_0 +X_1$ be such that there is an $f \in \ms{H}(\mathbb{S};X_0+X_1)$ with
 $f(\theta) = x$, $f_\theta \in L^1(\R;X_0+X_1)$ and
 \begin{equation*}
  t \mapsto \ee^{t(j-\theta)} \widehat{f}_\theta(t) \in L^{p_j}(\R;X_j),\qquad j=0,1.
 \end{equation*}
   Since $\widehat{f}_\theta \in L^1(\R;X_0+X_1)$ by H\"older's inequality, we have by Fourier inversion (see \cite[Proposition 2.4.5]{HNVW16}) $$\frac{1}{{2\pi}}\int_\R \widehat{f}_\theta(t)\dd t = f_\theta(0) =  x$$ with convergence in $X_0+X_1$. Thus using $g= \frac{\widehat{f_\theta}}{{2\pi}}$ in the definition of $({X}_0,{X}_1)_{\theta,p_0,p_1}$
   we have
   \begin{align*}
     \nrm{x}_{({X}_0,{X}_1)_{\theta,p_0,p_1}} \leq  \frac{1}{{2 \pi}} \max_{j = 0,1}\,\nrm{t \mapsto \ee^{t(j-\theta)} \widehat{f}_\theta(t)}_{L^{p_j}(\R;X_j)}.
   \end{align*}
   Taking the infimum over all such $f$ yields the inequality ``$\lesssim$''.

Conversely take $x \in ({X}_0,{X}_1)_{\theta,p_0,p_1}$ and let $f\in \ms{S}(\R;X_0+X_1)$ be such that  $\int_{\R}f(t)\dd t = x$ and
 \begin{equation*}
  t \mapsto \ee^{t(j-\theta)} f(t) \in L^{p_j}(\R;X_j),\qquad j=0,1.
 \end{equation*}
Then, using H\"older's inequality, we note that
\begin{equation*}
  g(z):= \int_{\R}   f(t) \ee^{(z-\theta)t}\dd t , \qquad z \in {\mathbb{S}}
\end{equation*}
is absolutely convergent in $X_0+X_1$ for any $z \in \mathbb{S}$ and thus $g \in \ms{H}(\mathbb{S};X_0+X_1)$.  Moreover $$g_\theta ={2 \pi}\, \widecheck{f} \in \ms{S}(\R;X_0+X_1)\subseteq L^1(\R;X_0 + X_1)$$
and $g(\theta)= x$. Taking the infimum over all such $f$ yields the inequality ``$\gtrsim$'' by Lemma \ref{lemma:schwartzenough}\ref{it:schwartzLionPeetre}.
\end{proof}

%For a Banach space  $X$ let
%$\ms{H}_0(\overline{\mathbb{S}};X)$
%be the set of all $f \in \ms{H}({\mathbb{S}};X)$ which extend to a continuous function $f:\overline{\mathbb{S}} \to X$ such that $$F(t):= \sup_{s \in [0,1]} \nrm{f(s+it)}_{X}, \qquad t \in \R,$$ satisfies $F \in L^1(\R)$ and $\lim_{\abs{t} \to \infty} F(t)= 0$.
%
%\begin{lemma}
%  Let $X$ be a Banach space and $0<s_0<s_1<1$. For $f \in \ms{H}^1(\mathbb{S};X)$ we have
%  \begin{equation*}
%    \lim_{\abs{t}\to \infty} \sup_{s \in [s_0,s_1]} \nrm{f(s+it)}_X = 0.
%  \end{equation*}
%\end{lemma}
%
%\begin{proof}
%  Take $s_0' \in (0,s_0)$ and $s_1' \in (s_1,1)$.
%  Then we know by \cite[Proposition H.1.3]{HNVW17} that $$\sup_{s_0'<\re z<s_1'} \nrm{f(z)}_X<\infty.$$ Therefore $g\colon \overline{\mathbb{S}} \to X$ given by $g(z) = f(s_0' +(s_1'-s_0')z)$ is a bounded analytic function, so by \cite[Lemma C.2.9]{HNVW16} we have
%  \begin{equation*}
%g(z) = \int_\R g_0(u)P_0(z;u)\dd u+ \int_\R g_1(u)P_1(z;u)\dd u, \qquad z \in \mathbb{S},
%\end{equation*}
%where $P_0$ and $P_1$ are the Poisson kernels on the strip. Note that $g_0,g_1 \in L^1(\R;X)$ and for $s_0'' = \frac{s_0-s_0'}{s_1'-s_0'}$, $s_0'' = \frac{s_1-s_0'}{s_1'-s_0'}$ and $u \in \R$ we have
% \begin{equation*}
% \lim_{\abs{t} \to \infty}  \sup_{s \in [s_0'',s_1'']} \abs{P_j(s+it;u)} = 0.
%\end{equation*}
%Therefore, by the dominated convergence theorem, we have
%   \begin{equation*}
%  \lim_{\abs{t} \to \infty}  \sup_{s \in [s_0,s_1]} \nrm{f(s+it)}_X = \lim_{\abs{t} \to \infty}  \sup_{s \in [s_0'',s_1'']} \nrm{g(s+it)}_X = 0,
%\end{equation*}
%proving the lemma.
%\end{proof}

In order to reformulate Proposition \ref{proposition:complexsimple} further in the spirit of the complex interpolation method, we will use that $t \mapsto \ee^{t(j-\theta)} \widehat{f}_\theta(t)$ is independent of $\theta \in (0,1)$ for $f \in\ms{H}^1({\mathbb{S}};X_0+X_1)$. Combined with the approximation in Lemma~\ref{lemma:schwartzenough}\ref{it:CcLionPeetre} this yields our complex formulation of the Lions--Peetre mean method.

\begin{theorem}\label{theorem:complex}
  Let $(X_0,X_1)$ be an interpolation couple of Banach spaces, let $p_0,p_1 \in [1,\infty]$ and let $\theta \in (0,1)$.
We have
\begin{align*}
  \nrm{x}_{({X}_0,{X}_1)_{\theta,p_0,p_1}} &\leq \frac{1}{2\pi}\inf\,\max_{j = 0,1}\,\nrm{\widehat{f}_j}_{L^{p_j}(\R;X_j)}, && x \in X_0+X_1,
\intertext{where the infimum is taken over all $f \in \ms{H}^1({\mathbb{S}};X_0+X_1)$ with $f(\theta) =x$.
Furthermore, if $\breve{X}$ is a dense subspace of $X_0\cap X_1$, we have}
  \nrm{x}_{({X}_0,{X}_1)_{\theta,p_0,p_1}} &\gtrsim_\theta \inf\,\max_{j = 0,1}\,\nrm{\widehat{f}_j}_{L^{p_j}(\R;X_j)}, &&x \in \breve{X},
\end{align*}
where the infimum is taken over all $f \in \ms{H}^1(\overline{\mathbb{S}};X_0+X_1)$ such that $f(\theta) =x$ and $(s,t)\mapsto \widehat{f}_s(t) \in C_c^\infty([0,1]\times \R)\otimes \breve{X}$.
\end{theorem}

\begin{proof}
Let $x \in X_0+X_1$ and take $f \in \ms{H}^1({\mathbb{S}};X_0+X_1)$ such that $f(\theta)=x$. For $s \in (0,1)$ and $n \in \N$ define
  \begin{equation*}
    g_{s,n}(\xi):= \int_{-n}^n f(s+it) \ee^{-(s+it)\xi} \dd t, \qquad \xi \in \R
  \end{equation*}
  and note that for $0<s_1<s_2<1$ we have by the Cauchy--Goursat theorem
  \begin{equation*}
    g_{s_1,n}(\xi)-g_{s_2,n}(\xi) = \sum_{\epsilon = \pm 1} \epsilon \int_{s_1}^{s_2} f(s+i\epsilon n)\ee^{-(s+i\epsilon n)\xi} \dd s, \qquad \xi \in \R.
  \end{equation*}
By \cite[Lemma H.1.4]{HNVW17} the right hand-side tends to zero for $n \to \infty$, so
\begin{equation*}
  \ee^{-s_1\xi}\widehat{f}_{s_1}(\xi) = \ee^{-s_2\xi}\widehat{f}_{s_2}(\xi), \qquad \xi\in \R.
\end{equation*}
Using the $L^1$-convergence of $f_s \to f_j$ for $j=0,1$, we deduce
  \begin{equation*}
 \widehat{f}_{j}(\xi) = \ee^{(j-\theta)\xi }\widehat{f}_{\theta }(\xi), \qquad \xi\in \R.
\end{equation*}
By Proposition \ref{proposition:complexsimple} (and an inspection of the proof for the constant) this implies
\begin{align*}
  \nrm{x}_{({X}_0,{X}_1)_{\theta,p_0,p_1}} &\leq \frac{1}{2\pi}\max_{j = 0,1}\,\nrm{ \widehat{f}_j}_{L^{p_j}(\R;X_j)},
\end{align*}
so taking the infimum over all such $f$ yields the first claim.

For the second claim take $x \in \breve{X}$ and let $g \in C_c^\infty(\R)\otimes \breve{X}$ with $\int_{\R}g(t) \dd t = x$. For
\begin{equation*}
  f(z): = \int_\R \ee^{\tau(z-\theta)}g(\tau) \dd \tau, \qquad z \in \overline{\mathbb{S}}
\end{equation*}
we have $f \in \ms{H}^1(\overline{\mathbb{S}}; X_0+X_1)$, $f(\theta)=x$ and $$(s,t)\mapsto \widehat{f}_s(t) \in C_c^\infty([0,1]\times \R)\otimes \breve{X}.$$ Moreover
 $\widehat{f}_j(t) = 2\pi \, \ee^{t(j-\theta)}g(t)$ for $t \in \R$ by Fourier inversion (see \cite[Proposition 2.4.5]{HNVW16}), so
\begin{equation*}
  \max_{j=0,1}\nrm{\widehat{f}_j}_{L^{p_j}(\R;X_j)} = 2\pi\,\max_{j=0,1} \,\nrm{t\mapsto \ee^{t(j-\theta)}g(t)}_{L^{p_j}(\R;X_j)}.
\end{equation*}
Taking the infimum over all such $f$ proves the second claim by Lemma \ref{lemma:schwartzenough}\ref{it:CcLionPeetre}.
\end{proof}

\begin{remark}\label{remark:distributions}
  As discussed in the introduction, Theorem \ref{theorem:complex} can be extended to a norm equivalence for all $x \in (X_0,X_1)_{\theta,p_0,p_1}$ using appropriate spaces of distributions, which will be further explored in \cite{LL21b,LL21a}.
\end{remark}

\section{Stein interpolation}\label{section:stein}
Using the complex formulation from the previous section we can now prove Stein interpolation for the Lions--Peetre method, and thus for the real interpolation method.

\begin{theorem}\label{theorem:Steininterpolation1}
Let $(X_0,X_1)$ and $(Y_0,Y_1)$ be interpolation couples of Banach spaces and let $\breve{X}$ be a dense subspace of $X_0\cap X_1$. Let $p_0,p_1,q_0,q_1 \in [1,\infty]$. Let $\cbrace{T(z)}_{z \in \overline{\mathbb{S}}}$ be a family of linear operators from $\breve{X}$ to $Y_0+Y_1$ such that
\begin{enumerate}[(i)]
  \item\label{it:stein1} $ T(\cdot) x\in \ms{H}(\overline{\mathbb{S}};Y_0+Y_1)$ for all $x \in \breve{X}$.
  %\item \label{it:stein1b} $\lim_{s \to j} T_{s+it}x = T_{j+it}x$ in $Y_0+Y_1$ for all $x \in \breve{X}$ and $j=0,1$.
\item \label{it:stein2} $T_j x:= T(j+i\cdot)x \in L^\infty(\R;Y_j)$  for all $x \in \breve{X}$ and
  \begin{equation*}
    \nrmb{(T_j \widehat{g}\hspace{2pt})^{\vee}}_{L^{q_j}(\R;Y_j)} \leq M_j \, \nrm{g}_{L^{p_j}(\R;X_j)}, \qquad g \in C_c^\infty(\R)\otimes \breve{X}.
  \end{equation*}
  for some $M_j>0$ and $j=0,1$.
\end{enumerate}
Then we have $T(\theta)\breve{X} \subseteq (Y_0,Y_1)_{\theta,q_0,q_1}$ for any $\theta \in (0,1)$ with
\begin{equation*}
  \nrm{T(\theta) x}_{(Y_0,Y_1)_{\theta,q_0,q_1}} \lesssim_\theta  M_0^{1-\theta} M_1^\theta \, \nrm{x}_{(X_0,X_1)_{\theta,p_0,p_1}}, \qquad x \in \breve{X}.
\end{equation*}
\end{theorem}

\begin{proof}
  Let $x \in \breve{X}$. By Theorem \ref{theorem:complex} we can find an $f \in \ms{H}^1(\overline{\mathbb{S}};X_0+X_1)$ such that $f(\theta)=x$, $(s,t)\mapsto \widehat{f}_s(t) \in C_c^\infty([0,1]\times \R)\otimes \breve{X}$ and
  \begin{equation*}
    \max_{j = 0,1}\,\nrm{\widehat{f}_j}_{L^{p_j}(\R;X_j)} \lesssim_\theta \nrm{x}_{(X_0,X_1)_{\theta,p_0,p_1}}.
  \end{equation*}
Note that $f \in \ms{H}^1(\mathbb{S})\otimes \breve{X}$, so $h(z) := \hab{\frac{M_0}{M_1}}^{z-\theta} T(z) f(z)$ is well-defined for $z \in \overline{\mathbb{S}}$.  Then $h(\theta) = T(\theta) x$ and by assumptions  \ref{it:stein1},  \ref{it:stein2} and Hadamard's three lines lemma we have $h \in \ms{H}^1(\overline{{\mathbb{S}}};Y_0+Y_1)$. Therefore, using assumption \ref{it:stein2} with $g = \frac{1}{2\pi}\widecheck{f_j}$ and Theorem \ref{theorem:complex}, we have
  \begin{align*}
   \nrm{T_\theta x}_{(Y_0,Y_1)_{\theta,q_0,q_1}} &\leq \max_{j=0,1} \, \nrm{\widehat{h}_j}_{L^{q_j}(\R;Y_j)}\\
    &= \max_{j=0,1}\has{\frac{M_0}{M_1}}^{j-\theta}  \nrmb{ \widehat{T_j f_j} }_{L^{q_j}(\R;Y_j)}\\
    &\leq \frac{1}{2\pi}\max_{j=0,1} \has{\frac{M_0}{M_1}}^{j-\theta} M_j \, \nrm{\widehat{f}_j}_{L^{p_j}(\R;X_j)}\\
    &\lesssim_\theta  M_0^{1-\theta} M_1^\theta\,\nrm{x}_{(X_0,X_1)_{\theta,p_0,p_1}},
  \end{align*}
  which finishes the proof.
\end{proof}

\begin{remark}~
\begin{enumerate}[(i)]
\item When $(p_0,p_1)\neq (\infty,\infty)$ in Theorem \ref{theorem:Steininterpolation1}, we can extend $T(\theta)$ to a bounded operator from $(X_0,X_1)_{\theta,p_0,p_1}$ to $(Y_0,Y_1)_{\theta,q_0,q_1}$ by density. For the case $(p_0,p_1)= (\infty,\infty)$ one could first extend the complex formulation to any $x \in (X_0,X_1)_{\theta,p_0,p_1}$ (see also Remark~\ref{remark:distributions}) to avoid a density argument.
  \item When $p_j=q_j$, the boundedness assumption on $T_j(\cdot)x$ in Theorem \ref{theorem:Steininterpolation1} follows from the assumption that  $T_j$ is an $L^{p_j}(\R;X_j)$-Fourier multiplier for $j=0,1$ (see \cite[Theorem 5.3.15]{HNVW16}).  When $p_j\neq q_j$ one can have Fourier multipliers which are not bounded, see e.g. \cite{Ho60,RV18}. One can relax the boundedness assumption on $T_j(\cdot)x$ in Theorem~\ref{theorem:Steininterpolation1}, as long as one ensures that $h \in \ms{H}^1({{\mathbb{S}}};Y_0+Y_1)$ and $h_s \to h_j$ for $s \to j$ in $L^1(\R;Y_0+Y_1)$ for $j=0,1$.
  \item The implicit constant in the conclusion of Theorem \ref{theorem:Steininterpolation1} is an artefact of the approximation in Lemma \ref{lemma:schwartzenough}. One could circumvent this approximation by going to a suitable class of distributions rather than functions, which we leave to the interested reader.
\end{enumerate}
\end{remark}

\section{Applications}\label{section:examples}
As a first, simple application of Theorem \ref{theorem:Steininterpolation1}, we will deduce the interpolation of weighted $L^p$-spaces under the real interpolation method from the unweighted case. For complex interpolation this argument is standard and for the $\gamma$-interpolation method this can be found in \cite[Theorem 3.2]{SW06}, but for real interpolation one usually employs different arguments (see e.g. \cite[Theorem IV.2.11]{KPS82}).

Let $X$ be a Banach space and let $(S,\mu)$ be a measure space. A measurable function $w\colon S \to (0,\infty)$ is called a weight and for $p \in [1,\infty)$ we define $L^p_w(S;X)$ as the space of all $f \in L^0(S;X)$ such that $fw \in L^p(S;X)$ with norm
\begin{equation*}
  \nrm{f}_{L^p_w(S;X)}:= \nrm{fw}_{L^p(S;X)}.
\end{equation*}

\begin{proposition}\label{proposition:weightedLp}
  Let $(X_0,X_1)$ be an interpolation couple of Banach spaces, let $(S,\mu)$ be a measure space and let $\theta \in (0,1)$. Take $p_0,p_1 \in [1,\infty]$ with $(p_0,p_1) \neq (\infty,\infty)$ and let $w_0,w_1\colon S \to (0,\infty)$ be weights. Set $\frac{1}{p} = \frac{1-\theta}{p_0}+\frac{\theta}{p_1}$ and let $w= w_0^{1-\theta } w_1^{{\theta}}$. Then we have
  \begin{align*}
    (L^{p_0}_{w_0}(S;X_0),L^{p_1}_{w_1}(S;X_1))_{\theta,p} = L^{p}_w(S;(X_0,X_1)_{\theta,p}).
  %  (L^{p_0}(S,w_0;X_0),L^{p_0}(S,w_0;X_0))_{\theta,\gamma} = L^{p}(S,w;(X_0,X_1)_{\theta,\gamma}).
  \end{align*}

\end{proposition}

\begin{proof}
Set $\breve{X} := L^{p_0}_{w_0}(S;X_0) \cap L^{p_1}_{w_1}(S;X_1)$ and for $z \in \overline{\mathbb{S}}$ define $$T(z) \colon \breve{X} \to L^{p_0}(S;X_0) + L^{p_1}(S;X_1)$$ by
  \begin{equation*}
    T(z) x(s) := w_0(s)^{1-z}w_1(s)^{z} \cdot x(s),\qquad s \in S.
  \end{equation*}
  Then $T(\cdot) x \in \ms{H}(\overline{\mathbb{S}};L^{p_0}(S;X_0) + L^{p_1}(S;X_1))$ is bounded for all $x \in \breve{X}$. Using Fubini's theorem, we have for $f \in C_c^\infty(\R)\otimes \breve{X}$
  \begin{align*}
    \nrmb{\ha{T_j\widehat{f}\hspace{2pt}}^{\vee} }_{L^{p_j}(\R;L^{p_j}(S;X_j))} &=
    \nrms{(s,t) \mapsto {f\has{t-\log\has{\frac{w_0(s)}{w_1(s)}},s} }}_{L^{p_j}_{w_j}(S;L^{p_j}(\R;X_j))}\\
    &= \nrm{f}_{L^{p_j}(\R;L_{w_j}^{p_j}(S;X_j))}
  \end{align*}
  for $j=0,1$.
  Therefore, by Theorem \ref{theorem:Steininterpolationintro},
  $$T(\theta) \colon (L^{p_0}_{w_0}(S;X_0),L^{p_1}_{w_1}(S;X_1))_{\theta,p} \to (L^{p_0}(S;X_0),L^{p_1}(S;X_1))_{\theta,p}$$
  is bounded. Since it is clearly injective and
  $$
  (L^{p_0}(S;X_0),L^{p_1}(S;X_1))_{\theta,p} = L^{p}(S;(X_0,X_1)_{\theta,p})
  $$
  by \cite[Theorem 2.2.10]{HNVW16}, this proves the embedding ''$\hookrightarrow$''. The converse embedding is proven similarly.
\end{proof}

In Proposition \ref{proposition:weightedLp} we were able to check the Fourier multiplier assumption in Theorem \ref{theorem:Steininterpolationintro} using Fubini's theorem. In general one needs to use a Fourier multiplier theorem to check this assumption. One could for example use the operator-valued Mihlin multiplier theorem (see \cite[Theorem 5.3.18]{HNVW16}), but this leads to undesirable restrictions on the involved Banach spaces in applications. Fortunately, we can often use the flexibility in our choice for $\cbrace{T(z)}_{z \in \overline{\mathbb{S}}}$ to simplify the Fourier multiplier assumption in Theorem~\ref{theorem:Steininterpolation1} considerably for smooth $T_j$.
As an example we will interpolate the angle of {($\mc{R}$-)sectoriality} of a ($\mc{R}$-)sectorial operator using the real interpolation method, which for the complex interpolation method was done in \cite[Corollary 3.9]{KKW06}. This result will allow us to improve a result of Fackler \cite{Fa13b} on the extrapolation of the analyticity of a semigroup on the real interpolation scale.

Let $X$ be a Banach space, let $A$ be a closed operator on $X$ and for $\sigma \in (0,\infty)$ we define the sector
\begin{equation*}
  \Sigma_\sigma:= \cbraceb{z \in \C\setminus \cbrace{0}:\abs{\arg(z)}<\sigma}.
\end{equation*}
We say that $A$ is \emph{sectorial} if there is a $\sigma \in (0,\pi)$ such that
\begin{equation*}
 \sup_{z \in \C\setminus \overline{\Sigma}_\sigma} \nrm{zR(z,A)} <\infty,
 \end{equation*}
and we denote the infimum over all such $\sigma$ by $\omega(A)$. We say that $A$ is \emph{$\mc{R}$-sectorial} if there is a $\sigma \in (0,\pi)$ such that
\begin{equation*}
 \mc{R}(\cbrace{zR(z,A):z \in \C\setminus \overline{\Sigma}_\sigma}) <\infty,
\end{equation*}
where $\mc{R}(\Gamma)$ denotes the $\mc{R}$-bound of an operator family $\Gamma \subseteq \mc{L}(X)$ (see \cite[Chapter 8]{HNVW17}. We denote the infimum over all such $\sigma$ by $\omega_{\mc{R}}(A)$. For an introduction to ($\mc{R}$-)sectorial operators we refer to \cite[Chapter 10]{HNVW17} and the references therein.

\begin{proposition}\label{proposition:interpolationsectorial}
  Let $(X_0,X_1)$ be an interpolation couple of Banach spaces and $p \in [1,\infty)$. Let $A_\theta$ be a closed operator on $(X_0,X_1)_{\theta,p}$ for all $\theta \in [0,1]$  such that the following consistency assumption is satisfied:
  \begin{equation*}
    R(z,A_\theta)x = R(z,A_{\theta'})x, \qquad x \in X_0\cap X_1, \, \theta,\theta' \in [0,1],\, z \in \rho(A_\theta)\cap \rho(A_{\theta'}).
  \end{equation*}
  For $\theta \in (0,1)$ we have:
  \begin{enumerate}[(i)]
    \item \label{it:sect1} If $A_0$ and $A_1$ are sectorial, then $A_\theta$ is sectorial with $$\omega(A_\theta) \leq (1-\theta)\omega(A_0)+\theta \omega(A_1).$$
    \item \label{it:sect2} If $X_0$ and $X_1$ have nontrivial type and $A_0$ and $A_1$ are $\mc{R}$-sectorial, then $A_\theta$ is $\mc{R}$-sectorial with $$\omega_{\mc{R}}(A_\theta) \leq (1-\theta)\omega_{\mc{R}}(A_0)+\theta \omega_{\mc{R}}(A_1).$$
  \end{enumerate}
\end{proposition}

\begin{proof}
Take $\omega(A_j)<\sigma_j<\pi$ for $j=0,1$ and let $\theta \in (0,1)$. Fix $s>0$ and for $z \in \overline{\mathbb{S}}$ define $T(z):X_0\cap X_1 \to X_0+X_1$ by
  \begin{equation*}
    T(z)x := \ee^{(z-\theta)^2}\cdot s \ee^{i(\sigma_0+(\sigma_1-\sigma_0) z)} R(s \ee^{i(\sigma_0+(\sigma_1-\sigma_0) z)},A_\theta)x.
  \end{equation*}
By the consistency assumption and the analyticity of $zR(z,A_j)$ on $X_j$ for $j=0,1$ we have
$T(\cdot)x \in \ms{H}(\overline{\mathbb{S}};X_0+X_1)$ for any $x \in X_0 \cap X_1$. Define $$f_j(t):=s \ee^{i(\sigma_0+(\sigma_1-\sigma_0) (j+it))}, \qquad t \in \R,\,j=0,1.$$ Then we have for $n \in \cbrace{0,1,2}$ and $t\in \R$
\begin{align*}
  \nrms{\frac{\ddn^n}{\ddn t^n} f_j(t) R(f_j(t),A_j)}_{\mc{L}(X_j)} &\leq \abs{\sigma_1-\sigma_0}^n \sum_{k=0}^n \nrms{f_j(t)^k R(f_j(t),A_j)^k}_{\mc{L}(X_j)}\leq C_j
\end{align*}
with $C_j>0$ independent of $s,t,\sigma_0,\sigma_1$.
Therefore, setting $$T_j(t) :=  \ee^{(j+it-\theta)^2}f_j(t) R(f_j(t),A_j), \qquad t \in \R,$$ we obtain
\begin{align*}
 % \nrm{T_j}_{L^1(\R;\mc{L}(X_j))} &\lesssim  M_{\sigma, A_j} \,\int_{\R} \ee^{(j-\theta)^2-t^2} \dd t \lesssim  M_{\sigma, A_j},\\
  \nrms{t\mapsto \frac{\ddn^n}{\ddn t^n}T_j(t)}_{L^1(\R;\mc{L}(X_j))} &\lesssim  C_j\sum_{k=0}^n \int_{\R} \abs{t}^k \ee^{(j-\theta)^2-t^2} \dd t \lesssim  C_j.
\end{align*}
Defining $k_j:=\widecheck{T_j}$, we have
\begin{align*}
  \nrm{k_j}_{L^1(\R;\mc{L}(X_j))} &\leq \nrmb{t\mapsto \frac{1}{1+t^2}}_{L^1(\R)} \nrmb{t\mapsto (1+t^2)k_j(t)}_{L^\infty(\R;\mc{L}(X_j))}\\
  &\lesssim \nrm{T_j}_{L^1(\R;\mc{L}(X_j))}+ \nrms{t\mapsto \frac{\ddn^2}{\ddn t^2}T_j(t)}_{L^1(\R;\mc{L}(X_j))}\lesssim  C_j.
\end{align*}
We conclude that $T_j,k_j \in L^1(\R;\mc{L}(X_j))$, so we have by Fourier inversion and Young's inequality for $f \in \ms{S}(\R; X_j)$
\begin{align*}
  \nrmb{(T_j\widehat{f}\hspace{2pt})^{\vee}}_{L^{p}(\R;X_j)} =  \nrm{k_j*f}_{L^{p}(\R;X_j)} \lesssim  C_j\, \nrm{f}_{L^{p}(\R;X_j)}.
\end{align*}
Thus, using the consistency assumption, we conclude that $\cbrace{T(z)}_{z \in \overline{\mathbb{S}}}$ satisfies the assumptions of Theorem \ref{theorem:Steininterpolationintro}, so we deduce for $x \in X_0\cap X_1$ that
\begin{align*}
  \nrm{s\ee^{(1-\theta)\sigma_0+\theta \sigma_1}) & R(s\ee^{(1-\theta)\sigma_0+\theta \sigma_1},A_\theta)}_{(X_0,X_1)_{\theta,p}\to (X_0,X_1)_{\theta,p}} \\&= \nrm{T(\theta)}_{(X_0,X_1)_{\theta,p}\to (X_0,X_1)_{\theta,p}} \lesssim_\theta C_0^{1-\theta} C_1^\theta.
\end{align*}
  An analogous statement holds when $\sigma_j$ is replaced by $-\sigma_j$, so we conclude that $A_\theta$ is sectorial with
$$\omega(A_\theta) \leq (1-\theta)\sigma_0+\theta \sigma_1.$$
Taking $\sigma_j$ abitrary close to $\omega(A_j)$ for $j=0,1$ finishes the proof of \ref{it:sect1}.

For \ref{it:sect2} let $X$ be a Banach space and let $(\varepsilon_n)_{n=1}^\infty$ be a sequence of independent Rademacher variables on a probability space $(\Omega,\P)$. Let $\varepsilon(X)$ be the Banach space of sequences $(x_k)_{k =1}^\infty$ such that $\sum_{k=1}^\infty {\varepsilon_kx_k}$ converges in $L^2(\Omega;X)$, endowed with the norm
\begin{equation*}
  \nrm{(x_n)_{k=1}^\infty}_{\varepsilon(X)}:= \nrms{\sum_{k=1}^\infty \varepsilon_kx_k}_{L^2(\Omega;X)}.
\end{equation*}
We note that a closed operator $B$ on $X$ is $\mc{R}$-sectorial if and only if
\begin{align*}
  \widetilde{B}(x_k)_{k=1}^\infty &:= (Bx_k)_{k=1}^\infty\\
  D(\widetilde{B}) &:= \cbraceb{(x_k)_{k=1}^\infty \in \varepsilon(D(B))}
\end{align*}
defines a sectorial operator on $\varepsilon(X)$.
 Moreover we have $\omega_{\mc{R}}(B) = \omega(\widetilde{B}).$
Therefore,  \ref{it:sect2} follows from  \ref{it:sect1} and the isomorphism
\begin{equation*}
  (\varepsilon(X_0),\varepsilon(X_1))_{\theta,p} = \varepsilon((X_0,X_1)_{\theta,p}),
\end{equation*}
see \cite[Proposition 6.3.1, Theorem 7.4.16 and Theorem 7.4.23]{HNVW17}.
\end{proof}

\begin{remark}
In \cite{HHK06} Haak, Haase and Kunstmann interpolate the $\mc{R}$-sectoriality of $A$ on $X_0$ with the $\mc{R}$-boundedness of $\cbrace{R(\lambda,A), \lambda>0}$ on $X_1$. Using Stein interpolation for the complex interpolation method on the analytic operator family $\lambda^z R(\lambda,A)$, one obtains the $\mc{R}$-boundedness of $\cbrace{\lambda^{1-\theta}R(\lambda,A), \lambda>0}$ on $[X_0,X_1]_\theta$. The conditions for Stein interpolation for the complex interpolation method, which are pointwise multiplier conditions, are easily checked using randomization (see \cite[Proposition 6.1.11]{HNVW17}).
 However, when one tries to apply  Theorem \ref{theorem:Steininterpolationintro} for the real interpolation method, the Fourier multiplier conditions turn $\lambda^{it}$ into a translation operator. It is well known that the family of translation operators is not $\mc{R}$-bounded on $L^p(\R;X)$ (see \cite[Proposition 8.1.16]{HNVW17}) and thus the argument fails for real interpolation. This provides some intuition on why the counterexample provided in \cite[Example 6.13]{HHK06} works.
\end{remark}

Let $(S(t))_{t \geq 0}$ be a semigroup of bounded operators on $X$. We say that $(S(t))_{t \geq 0}$ is a \emph{$C_0$-semigroup} if $$\lim_{t \downarrow 0} \nrm{S(t)x-x}_X = 0, \qquad x \in X.$$
We define the generator of a $C_0$-semigroup $(S(t))_{t \geq 0}$ as the closed operator $Ax:= \lim_{t\downarrow 0} \frac{S(t)x-x}{t}$ with as domain the set of all $x \in X$ for which this limit exists.

We say that$(S(t))_{t \geq 0}$ is \emph{analytic} if for some $\sigma \in (0,\frac{\pi}{2})$ the function  $t \mapsto S(t)x$ extends analytically to $\Sigma_\sigma$  for all $x \in X$  and call supremum over such $\sigma$ the angle of analyticity. We say that $(S(t))_{t \geq 0}$ is \emph{$\mc{R}$-analytic} if it is analytic and there is a $\sigma' \in (0,\sigma]$ such that
\begin{equation*}
  \mc{R}(\cbrace{T(z):z \in \Sigma_\sigma,\,\abs{z}\leq 1})<\infty
\end{equation*}
and we call the supremum over such $\sigma$ the angle of $\mc{R}$-analyticity.
We say that $(S(t))_{t \geq 0}$ is a \emph{bounded} ($\mc{R}$-)analytic semigroup if it is ($\mc{R}$-)analytic and $\cbrace{S(z):z \in \Sigma_\sigma}$ is ($\mc{R}$-)bounded for all $\sigma$ smaller than the angle of ($\mc{R}$-) analyticity.

 %We refer to \cite[Section II.4.a]{EN00} and \cite[Appendix G]{HNVW17} for an introduction to analytic $C_0$-semigroups and their relation to sectorial operators

It is well-known that a densely defined ($\mc{R}$-)sectorial operator $A$ on a Banach space $X$ with $\omega(A)<\frac{\pi}{2}$ generates a bounded ($\mc{R}$-)analytic $C_0$-semigroup $(\ee^{-tA})_{t\geq 0}$ of angle $\frac{\pi}{2}-\omega(A)$ and vice versa that the negative generator of a bounded ($\mc{R}$-)analytic $C_0$-semigroup is a densely defined \mbox{($\mc{R}$-)}secto\-rial operator (see \cite[Theorem 4.6]{EN00} and \cite[Proposition 10.3.3]{HNVW17}).
Therefore, Proposition \ref{proposition:interpolationsectorial} yields new information on ($\mc{R}$-)analytic semigroups on the real interpolation scale. In particular, we can quantify the angle of ($\mc{R}$-)analyticity in a result of Fackler \cite{Fa13b} (see also \cite{Fa15b}).

\begin{corollary}\label{corollary:semigroup}
  Let $(X_0,X_1)$ be an interpolation couple of Banach spaces and let $(S(t))_{t \geq 0}$ be a semigroup on $X_0+X_1$ which leaves both $X_1$ and $X_2$ invariant. Fix $\theta \in (0,1)$ and $p \in [1,\infty)$.
  \begin{enumerate}[(i)]
    \item\label{it:semi1} If $(S(t))_{t \geq 0}$ is
   an analytic $C_0$-semigroup of angle $\sigma \in (0,\frac{\pi}{2})$ on $X_0$ and $\cbrace{S(t):t \in (0,1)}$ is uniformly bounded on $X_1$, then $(S(t))_{t \geq 0}$ is an analytic $C_0$-semigroup of angle at least $(1-\theta) \sigma$ on $(X_0,X_1)_{\theta,p}$.
    \item\label{it:semi2}  If $X_0$ and $X_1$ have nontrivial type,  $(S(t))_{t \geq 0}$ is
   a  $\mc{R}$-analytic $C_0$-semigroup of angle $\sigma \in (0,\frac{\pi}{2})$ on $X_0$ and $\cbrace{S(t):t \in (0,1)}$ is $\mc{R}$-bounded on $X_1$, then $(S(t))_{t \geq 0}$ is an $\mc{R}$-analytic $C_0$-semigroup of angle at least $(1-\theta) \sigma$ on $(X_0,X_1)_{\theta,p}$.
  \end{enumerate}
\end{corollary}

\begin{proof}
Take $\eta \in (\theta,1)$ and set $Y_1 := (X_0,X_1)_{\eta,p}$. By \cite[Theorem 3.3]{Fa13b} we know that $(S(t))_{t \geq 0}$ is an analytic $C_0$-semigroup on $Y_1$.
 Let $\omega>0$ be such that $\widetilde{S}(t):= \ee^{-\omega t}S(t)$ is a bounded analytic $C_0$-semigroup on both $X_0$ and $Y_1$ and note that the angles of analyticity of $(S(t))_{t \geq 0}$ and $(\widetilde{S}(t))_{t \geq 0}$ on $X_0$ and $Y_1$ are equal. Let $A_0$ and $A_1$ be the generators of $(\widetilde{S}(t))_{t \geq 0}$ on $X_0$ and $Y_1$ respectively, which by \cite[Theorem 4.6]{EN00} are sectorial operators with
 $$
 \omega(A_j) \leq \frac{\pi}{2}-(1-j)\sigma, \qquad j=0,1.
 $$
For  $\theta' \in (0,1)$ we note that $(\widetilde{S}(t))_{t \geq 0}$ is a $C_0$-semigroup on $(X_0,Y_1)_{\theta',p}$ and we denote its generator by $A_{\theta'}$. Then $(A_{\theta'})_{{\theta'} \in [0,1]}$ satisfies the assumptions of Proposition \ref{proposition:interpolationsectorial}, so we obtain that $A_{\theta'}$ is sectorial on $(X_0,Y_1)_{\theta',p}$ with
$$
\omega(A_{\theta'}) \leq \frac{\pi}{2} - (1-\theta')\sigma
$$
 and thus by \cite[Theorem 4.6]{EN00} we know that $(\widetilde{S}(t))_{t \geq 0}$ is a bounded analytic $C_0$-semigroup on $(X_0,Y_1)_{\theta',p}$ of angle at least $(1-\theta')\sigma$. It follows that $({S}(t))_{t \geq 0}$ is an analytic $C_0$-semigroup on $(X_0,Y_1)_{\theta',p}$ of angle at least $(1-\theta')\sigma$ and thus, using reiteration for real interpolation \cite[Theorem 3.5.3]{BL76} and taking $\eta$ arbitrary close to $1$, \ref{it:semi1} follows.  Moreover \ref{it:semi2} follows from \ref{it:semi1} by a similar argument as the one we used for Proposition \ref{proposition:interpolationsectorial}\ref{it:sect2}
\end{proof}

\subsection*{Acknowledgements}
The authors wish to thank Mark Veraar for his suggestion to include the application to ($\mc{R}$-)sectorial operators, \'Oscar Dom\'inguez Bonilla for his helpful references to the literature and Jan Rozendaal for his remarks on the introduction.

\bibliographystyle{alpha}
\bibliography{Steinbib}

\end{document}